\documentclass[12pt]{article}
\usepackage{amsmath,amssymb}

\newtheorem{theorem}{Theorem}[section]
\newtheorem{proposition}[theorem]{Proposition}

\newtheorem{lemma}[theorem]{Lemma}
\newtheorem{definition}[theorem]{Definition}
\newtheorem{example}[theorem]{Example}
\newtheorem{remark}[theorem]{Remark}

\newenvironment{proof}{\smallskip\par{\sc Proof.}\enspace}%
 {{\unskip\nobreak\hfil\penalty50\hskip2em
          \hbox{}\nobreak\hfil{\rule[-1pt]{5pt}{10pt}}
          \parfillskip=0pt\finalhyphendemerits=0
          \par\medskip}} 

\makeatletter
\def\section{\@startsection {section}{1}{\z@}{3.25ex plus 1ex minus
 .2ex}{1.5ex plus .2ex}{\large\bf}}
\def\subsection{\@startsection{subsection}{2}{\z@}{3.25ex plus 1ex minus
 .2ex}{1.5ex plus .2ex}{\normalsize\bf}}
\@addtoreset{equation}{section} 
\makeatother

\title{ On stochastic differential equations driven by the renormalized square of the Gaussian white noise}

\author{Bilel Kacem Ben Ammou\footnote{Department of Mathematics, Universtiy of Tunis - El Manar, Street Mohamed Alaya Kacem,  Nabeul - Tunisia. E-mail: \emph{bilelbenammou@gmail.com}}\quad\quad Alberto Lanconelli\footnote{Dipartimento di Matematica, Universit\'a degli Studi di Bari Aldo Moro, Via E. Orabona 4, 70125 Bari - Italia. E-mail: \emph{alberto.lanconelli@uniba.it}}}

\date{\empty}

\begin{document}

\maketitle

\numberwithin{equation}{section}

\bigskip

\begin{abstract}
We investigate the properties of the Wick square of Gaussian white noises through a new method to perform non linear operations on Hida distributions. This method lays in between the Wick product interpretation and the usual definition of nonlinear functions. We prove on It\^o-type formula  and solve stochastic differential equations driven by the renormalized square of the Gaussian white noise. Our approach works with standard assumptions on the coefficients of the equations, Lipschitz continuity and linear growth condition, and produces existence and uniqueness results in the space where the noise lives. The linear case is studied in details and positivity of the solution is proved.
\end{abstract}

Key words and phrases: Gaussian white noise, Hida distributions, stochastic differential equations, Wick product.

AMS 2000 classification: 60H40, 60H10.

\allowdisplaybreaks

\section{Introduction}
The Gaussian white noise $\{W_t\}_{t\geq 0}$ is a generalized process that can be formalized as the distributional time derivative of a standard one dimensional Brownian motion $\{B_t\}_{t\geq 0}$. One way to define its square is through the so-called Wick renormalization 
\begin{eqnarray}\label{wick square}
W_t^{\diamond 2}:=\lim_{\varepsilon\to 0}(W^{\varepsilon}_t)^2-E[(W^{\varepsilon}_t)^2]
\end{eqnarray}
where $W_t^{\varepsilon}$ is a smooth approximation of the white noise $W_t$ and the limit is interpreted in a suitable distributional sense. This object appears naturally in different contexts. For instance, in the paper \cite{Z2} the author considered gradient operators on the classical Wiener space along directions that do not belong to the Cameron-Martin space and obtained among other things the following integration by parts formula
\begin{eqnarray}\label{IBP}
E[D_BF(B)]=E\Big[F(B)\int_0^1W_t^{\diamond 2}dt\Big]
\end{eqnarray}
where $F(B)$ is a functional of the Brownian path belonging to a certain class of regularity and $D_B$ denotes differentiation along the direction $B$, i.e. $D_BF(B):=\lim_{\delta\to 0}\frac{F(B+\delta B)-F(B)}{\delta}$. One encounters quantities like (\ref{wick square}) also in connection with It\^o-type formulas for solutions to certain stochastic partial differential equations. More precisely, consider the  stochastic heat equation driven by an additive Gaussian space-time white noise (see \cite{W})
\begin{eqnarray}\label{SPDE}
\partial_t u(t,x)=\frac{1}{2}\partial_{xx}u(t,x)+W_{tx},\quad t>0,\quad x\in[0,1],
\end{eqnarray}
with initial condition $u(0,x)=0$ and homogeneous Dirichlet boundary conditions . It was proved in \cite{Z} (see also \cite{L3},\cite{L4}) that, if $\{u(t,x)\}_{t\geq 0,x\in [0,1]}$ denotes the unique (weak) solution to (\ref{SPDE}), then for any $l\in C_0^2(]0,1[)$ one has
\begin{eqnarray*}
\langle u(t,\cdot)^2,l\rangle&=&\frac{1}{2}\int_0^t\langle u(s,\cdot)^2,l''\rangle ds+2\int_0^t\langle u(s,\cdot),l dW_s\rangle+\langle E[u(t,\cdot)^2],l\rangle\\
&&-\int_0^t\langle(\partial_xu(s,\cdot))^{\diamond 2},l\rangle ds
\end{eqnarray*} 
where $\langle\cdot,\cdot\rangle$ denotes the inner product in $\mathcal{L}^2([0,1])$. The term $(\partial_xu(s,x))^{\diamond 2}$ in the previous equality is analogous to the one in (\ref{wick square}) since $\partial_xu(t,x)$ is a generalized Gaussian field.\\
The existence of limits of the type (\ref{wick square}) has also been considered in \cite{MR} where the authors investigated the almost sure existence of the limit
\begin{eqnarray*}
\lim_{\varepsilon\to 0}\int_0^{+\infty}\Big(\frac{X_{t+\varepsilon}-X_t}{\varepsilon}\Big)^{\diamond k}g(t)dt
\end{eqnarray*}
where $\{X_t\}_{t\geq 0}$ is a Gaussian process with a covariance satisfying certain assumptions (not satisfied by the standard Brownian motion), $k\geq 0$ and $g$ is a bounded and measurable function with compact support. Observe that for $X_t=B_t$, $k=2$ and $g=1_{[0,1]}$ the above limit would correspond to the integral in the right hand side of (\ref{IBP}).\\ 
We also mention the paper \cite{AFS} where the problem of defining powers of Gaussian white noises is taken from the point of view of quantum probability.\\
The analogy between (\ref{IBP}) and the usual integration by parts formula from the Malliavin calculus (\cite{N})
\begin{eqnarray*}
E[D_gF(B)]=E\Big[F(B)\int_0^1g'(s)dB_s\Big]
\end{eqnarray*}
where $g$ is a deterministic absolutely continuous function, suggests that the quantity $\int_0^1W_t^{\diamond 2}dt$ or more generally the process $t\mapsto\int_0^tW_s^{\diamond 2}ds$ plays in the differential calculus associated to the gradient $D_B$ (from (\ref{IBP})) the role played by $\int_0^tW_sds=B_t$ in the classical Malliavin calculus. There is however  a major difference in these situations: the fact that $\int_0^tW_s^{\diamond 2}ds$ is a generalized process (more precisely, a Hida distribution) and therefore the question of finding It\^o type formulas or studying stochastic differential equations driven by that process is far from being obvious.\\
Consider for instance equations of the type
\begin{eqnarray}\label{intro SDE}
\frac{dY_t}{dt}=b(Y_t)+\sigma(Y_t)\cdot W_t^{\diamond 2}
\end{eqnarray}
for suitable measurable coefficients $b$ and $\sigma$. First of all, note that for $b=0$ and $\sigma=1$ equation (\ref{intro SDE})   is solved by $Y_t=Y_0+\int_0^tW_s^{\diamond 2}ds$ which is a Hida distribution. Therefore, one cannot expect the solution to (\ref{intro SDE}) to be more regular than that. This fact implies the necessity to give a meaning to the nonlinear terms appearing in (\ref{intro SDE}) and to the multiplication between $\sigma(Y_t)$ and $W_t^{\diamond 2}$. One possibility is to interpret all the nonlinear terms in the Wick-product sense (see \cite {HOUZ} and the references quoted there), that means to replace equation (\ref{intro SDE}) with
\begin{eqnarray}\label{Wick SDE}
\frac{dZ_t}{dt}=b^{\diamond}(Z_t)+\sigma^{\diamond}(Z_t)\diamond W_t^{\diamond 2}
\end{eqnarray}
(we refer to the next sections for precise definitions). However, this procedure has at least two important drawbacks: firstly, to define $b^{\diamond}(Z_t)$ and $\sigma^{\diamond}(Z_t)$ one needs the analyticity of the functions $b$ and $\sigma$; secondly, the solution to (\ref{Wick SDE}) usually exists in spaces that are much bigger than the one where the noise lives (such as the Kondratiev spaces). Moreover, from a modeling point of view, solutions to equation (\ref{Wick SDE}) may exhibit behaviors that differs from what is expected to happen for solutions to (\ref{intro SDE}) (see for instance \cite{C}). \\
Our aim in this paper is to introduce a new method to define nonlinear operation on Hida distributions. This method lays in between the Wick product interpretation (\ref{Wick SDE}) and the usual definition of nonlinear functions (\ref{intro SDE}). Our approach requires standard assumptions on the coefficients of the stochastic equations considered (Lipschitz continuity and linear growth condition) and produces existence and uniqueness results in the space where the noise lives. \\
In the recent years, renormalization techniques for solving stochastic (partial) differential equations have attracted the attention of many authors (see \cite{DD}, \cite{EJS}, \cite{GIP}, \cite{H1}, \cite{H2} and the references quoted there). The common basic idea in these references is to smooth the noise, solve the corresponding equation and then try to compute the limit of the solution as the degree of regularization of the noise decreases (the existence of a non trivial limit usually requires a renormalization of the coefficients of the original equations). The way we treat nonlinear functions of distributions and solve related stochastic differential equations follows the same principle: the only technical difference is that we choose a specific regularizing procedure for the noise and then, once we have solved the regularized equation, we apply to it the inverse of the regularizing map utilized before (instead of letting a parameter tend to zero) (see Remark \ref{approximation idea} below). Our procedure deeply depends on the adopted smoothing map but this is intrinsically connected with the construction of the Hida distribution space which is the natural accommodation of the noise. We do not know whether the equations we consider can be handled with the recent theory of regularity structures (\cite{H2}); certainly, the points of view of the two approaches are different in the fact that our notion of distribution is related to the probability space where the noise is defined and not to the state space of the time parameter describing the processes.  \\
The paper is organized as follows: Section 2 is a quick review of the minimal background material needed to treat our problem; Section 3 introduces and describes our new method of performing non linear operations on Hida distributions while in Section 4 and 5 we apply these concepts to deduce an It\^o-type formula and a theorem on existence and uniqueness for solutions to stochastic differential equations driven by the renormalized square of the Gaussian white noise, respectively; finally, in Section 6 we propose few by-products of the introduced concepts.        

\section{Framework}
The aim of this section is to briefly set up the framework utilized to prove our main results. For more details we refer the interested reader to one of the books \cite{HKPS}, \cite{HOUZ}, \cite{Kuo} or to the paper \cite{DPV} where many technical issues are usefully spelled out.\\
Let $(\Omega,\mathcal{F},\mathcal{P})$ be a complete probability space endowed with a standard one dimensional Brownian motion $\{B_t\}_{t\geq 0}$ and denote by $\{\mathcal{F}_t\}_{t\geq 0}$ its augmented natural filtration. Write $\mathcal{G}$ to denote the smallest sigma-algebra containing all the $\mathcal{F}_t$'s for $t\geq 0$. According to the Wiener-It\^o chaos decomposition theorem any $X\in\mathcal{L}^2(\Omega,\mathcal{G},\mathcal{P})$ ($(L^2)$ for short) can be uniquely represented as
\begin{eqnarray}\label{WI}
X=\sum_{n\geq 0}I_n(h_n)
\end{eqnarray}
where $h_n\in\mathcal{L}^2(\mathbb{R}_+^n)$ is a symmetric function and $I_n(h_n)$ stands for the $n$-th order multiple It\^o integral of $h_n$ with respect to the Brownian motion $\{B_t\}_{t\geq 0}$. Observe that from (\ref{WI}) one gets
\begin{eqnarray*}
E[X^2]=\sum_{n\geq 0}n!|h_n|^2_{\mathcal{L}^2(\mathbb{R}_+^n)}
\end{eqnarray*}
(here $E[\cdot]$ denotes the expectation on the probability space $(\Omega,\mathcal{G},\mathcal{P})$). For $r\in [1,+\infty]$ we also set $(L^r):=\mathcal{L}^r(\Omega,\mathcal{G},\mathcal{P})$, the classic Lebesgue spaces over the measure space $(\Omega,\mathcal{G},\mathcal{P})$.\\
We are now going to introduce the Schwartz space over $\mathbb{R}_+$ and an analogous class of smooth random variables. Consider the differential operator
\begin{eqnarray*}
A:=-\frac{d}{dt}t\frac{d}{dt}+\frac{1}{4}t+1
\end{eqnarray*} 
acting on a subset of $\mathcal{L}^2(\mathbb{R}_+)$ and recall that for any $k\geq 0$ one has
\begin{eqnarray}\label{eigenvalues}
A\xi_k=\Big(k+\frac{3}{2}\Big)\xi_k
\end{eqnarray}
where $\{\xi_k\}_{k\geq 0}$ is the complete orthonormal system in $\mathcal{L}^2(\mathbb{R}_+)$ formed by the Laguerre functions. For $p\geq 0$ define $S_p$ to be the Hilbert space of functions in $\mathcal{L}^2(\mathbb{R}_+)$ such that
\begin{eqnarray*}
|f|_p:=|A^p f|_{\mathcal{L}^2(\mathbb{R}_+)}<+\infty
\end{eqnarray*}
(for consistency we will denote from now on the norm $|\cdot|_{\mathcal{L}^2(\mathbb{R}_+)}$ with the symbol $|\cdot|_0$). It is clear from (\ref{eigenvalues}) that for $p<q$ one has $S_q\subset S_p$; one can then prove that 
\begin{eqnarray*}
S:=\bigcap_{p\geq 0}S_p 
\end{eqnarray*} 
endowed with the projective limit topology, coincides with the Schwartz space over $\mathbb{R}_+$ of infinitely differentiable functions that vanish, together with all their derivatives, at $+\infty$ faster than any inverse power (see \cite{DPV}). The dual space of $S$, denoted by $S'$, is the space of tempered distributions over $\mathbb{R}_+$ and it can be represented as the union of the spaces $S_{-p}$ for $p\geq 0$. One of its most representative elements is the Dirac delta distribution $\delta_t$, $t\geq 0$. The dual pairing between $S'$ and $S$ will be denoted by $\langle\cdot,\cdot\rangle$.\\ 
\noindent We now lift this construction to the space $(L^2)$; more precisely, for $X=\sum_{n\geq 0}I_n(h_n)$ and $p\geq 0$ let
\begin{eqnarray*}
\Gamma(A^p)X:=\sum_{n\geq 0}I_n((A^{p})^{\otimes n}h_n)
\end{eqnarray*}
and define $(S_p)$ to be the Hilbert space of those $X$'s such that
\begin{eqnarray*}
\sum_{n\geq 0}n!|(A^p)^{\otimes n}h_n|^2_{\mathcal{L}^2(\mathbb{R}_+^n)}<+\infty.
\end{eqnarray*}
The space 
\begin{eqnarray*}
(S):=\bigcap_{p\geq 0}(S_p) 
\end{eqnarray*} 
endowed with the projective limit topology is called \emph{Hida test function space}. Its dual $(S)^*$, the \emph{Hida distribution space}, accommodates the white noise $W_t:=\frac{dB_t}{dt}$ which can be represented as $W_t:=I_1(\delta_t)$.  We will write $\langle\langle\cdot,\cdot\rangle\rangle$ for the dual pairing between $(S)^*$ and $(S)$.
For $f\in\mathcal{L}^2(\mathbb{R}_+)$ set
\begin{eqnarray*}
\mathcal{E}(f):=\exp\Big\{I_1(f)-\frac{|f|_0^2}{2}\Big\}.
\end{eqnarray*}
It is easy to prove that $\mathcal{E}(f)\in (S)$ for $f\in S$ and that
\begin{eqnarray*}
\langle\langle X,\mathcal{E}(f)\rangle\rangle=\langle\langle Y,\mathcal{E}(f)\rangle\rangle,\quad\mbox{ for all $f\in S$}
\end{eqnarray*}  
implies $X=Y$ in $(S)^*$. The map
\begin{eqnarray*}
f\mapsto (SX)(f):=\langle\langle X,\mathcal{E}(f)\rangle\rangle
\end{eqnarray*}
is called $S$-transform of $X\in (S)^*$. The celebrated characterization theorem (\cite{HKPS}) provides a necessary and sufficient condition for the invertibility of the $S$-transform: let $f\in S\mapsto F(f)\in\mathbb{R}$ be a measurable function such that
\begin{itemize}
\item for all $f,g\in S$ the function $x\in\mathbb{R}\to F(f+xg)$ has an entire analytic extension to the complex plane (denoted with the same symbol);
\item there exist positive constants $K,p$ such that for any $f\in S$ and $z\in\mathbb{C}$ one has $|F(zf)|\leq K\exp\{\frac{|z|^2}{2}|f|_{p}^2\}$.
\end{itemize}
Then there exists $X\in (S)^*$ such that $(S X)(f)=F(f)$. If these conditions are met for the function
\begin{eqnarray*}
f\mapsto\int_a^bS(X_t)(f)dt
\end{eqnarray*}
we will say that the process $\{X_t\}_{t\geq 0}$ is Pettis integrable in $(S)^*$ over the interval $[a,b]$; the value of the integral is denoted by $\int_a^bX_t dt$ and it verifies
\begin{eqnarray*}
\Big\langle\Big\langle\int_a^bX_t dt,\varphi\Big\rangle\Big\rangle=\int_a^b\langle\langle X_t,\varphi\rangle\rangle dt
\end{eqnarray*}
for all $\varphi\in (S)$. Let $\tau\geq 0$; a Hida distribution $X$ will said to be $\mathcal{F}_\tau$-measurable if
\begin{eqnarray}\label{def measurable}
(S X)(f+g)=(S X)(f)
\end{eqnarray}   
for all $f,g\in S$ such that the support of $g$ in contained in $[0,\tau]^c$ (\cite{DPV}). Finally, for $X,Y\in (S)^*$ we write $X\diamond Y$ for the unique element in $(S)^*$ such that 
\begin{eqnarray*}
\langle\langle X\diamond Y,\mathcal{E}(f)\rangle\rangle=\langle\langle X,\mathcal{E}(f)\rangle\rangle\cdot\langle\langle Y,\mathcal{E}(f)\rangle\rangle,
\end{eqnarray*}
for all $f\in S$. The quantity $X\diamond Y$ is named Wick product of $X$ and $Y$.

\section{A renormalized product} 

The following definition introduces the main tool utilized in the investigation presented in this paper.
\begin{definition}\label{main definition}
Let $\varphi:\mathbb{R}\to\mathbb{R}$ be a bounded function and let $X$ belong to  $(S_{-p})\subset
(S)^*$ for some $p\geq 0$. We define
\begin{eqnarray}\label{def nonlinear}
\tilde{\varphi}_{p}(X):=\Gamma(A^{p})(\varphi(\Gamma(A^{-p})X)).
\end{eqnarray}
\end{definition}
The idea behind this definition is clear: we take a Hida distribution $X$; by construction there exists a $p\geq 0$ such that $X\in (S_{-p})$ that means $\Gamma(A^{-p})X\in (L^2)$; now we apply the nonlinear function $\varphi$ to $\Gamma(A^{-p})X$, viewed as a smooth approximation of $X$ and then we "remove" the regularization by applying $\Gamma(A^p)$. Observe that here we cannot remove the regularization by taking the limit as $p$ tends to zero of $\varphi(\Gamma(A^{-p})X)$ since it wouldn't exists without some other renormalization.   \\
The definition of $\tilde{\varphi}_{p}(X)$ does indeed depend on $p$, as it is stressed in the notation. Since $(S_{-p})\subset (S_{-q})$ for $p<q$ we could have chosen a bigger value of the parameter and obtain a different renormalized object. The point is however to prefer the smallest possible exponent in $\Gamma(A^{-p})X$ in order to preserve $X$ as much as possible. If for instance $X\in (L^2)$, we do not need to perform any approximation and we can simply take $p=0$ and get the usual $\varphi(X)$. \\
The assumption of boundedness on $\varphi$ in Definition \ref{main definition} can be clearly relaxed for specific choices of $X$. In general, the quantity $\tilde{\varphi}_{p}(X)$ belongs to $(S_{-p})$, the same space of $X$.
\begin{definition}
Let $X,Y\in (S_{-q})\subset (S)^*$ for some $q\geq 0$: then for any $p>q$ we set 
\begin{eqnarray}\label{def product}
X\ast_p Y:=\Gamma(A^p)(\Gamma(A^{-p})X\cdot\Gamma(A^{-p})Y).
\end{eqnarray}
\end{definition}

\begin{remark}
If $X\in (S_{-q})$ then $\Gamma(A^{-p})X$ belongs to $(L^r)$ for some $r>2$. In fact
\begin{eqnarray*}
\Gamma(A^{-p})X=\Gamma(A^{-p+q})\Gamma(A^{-q})X;
\end{eqnarray*}
by definition $\Gamma(A^{-q})X\in (L^2)$; moreover, be the Nelson hyper-contractive estimate (\cite{Nelson}) the operator $\Gamma(A^{-p+q})$ maps $(L^2)$ into a smaller $(L^r)$ for some $r>2$. This is necessary to guarantee that the product $\Gamma(A^{-p})X\cdot\Gamma(A^{-p})Y$ belongs to $(L^u)$ for some $u>1$.
\end{remark}
The product defined in (\ref{def product}) is commutative, associative and distributive with respect to the sum. It was introduced in a slightly different form in \cite{DLS} in connection to Wong-Zakai approximation theorems and utilized subsequently in \cite{DLS2} in the study of certain generalizations of the Poincar\'e inequality. It is instructive to observe (see \cite{DLS}) that 
\begin{eqnarray*}
\lim_{p\to +\infty}X\ast_p Y=X\diamond Y\quad\mbox{ in }(S)^*.
\end{eqnarray*}  
Therefore the product $\star_p$ is collocated between the ordinary product (when $p=0$) and the Wick product (when $p=+\infty$).

\section{It\^o-type formula for the quadratic white noise process}

In this section we are going to prove a chain-rule formula for the object of our investigation, i.e. $\int_0^tW_s^{\diamond 2}ds$. We begin with the following technical lemma.

\begin{lemma}\label{technical lemma}
For $p>\frac{1}{2}$ the function $t\mapsto |A^{-p}\delta_t|_0^2$ is continuous. For $p>1$ the process $\{\Gamma(A^{-p})W_t\}_{t\geq 0}$ is almost surely continuous. 
\end{lemma}
\begin{proof}
Let $t\geq 0$; then $\delta_t\in S'$ and we can write
\begin{eqnarray*}
\delta_t=\sum_{k\geq 0}\xi_k(t)\xi_k.
\end{eqnarray*}
Then
\begin{eqnarray*}
A^{-p}\delta_t=\sum_{k\geq 0}\Big(k+\frac{3}{2}\Big)^{-p}\xi_k(t)\xi_k
\end{eqnarray*}
and hence
\begin{eqnarray*}
|A^{-p}\delta_t|_0^2=\sum_{k\geq 0}\Big(k+\frac{3}{2}\Big)^{-2p}\xi_k^2(t).
\end{eqnarray*}
Since $\sup_{t\in\mathbb{R}_+}|\xi_k(t)|=1$ (\cite{DPV}) we get that the previous series is uniformly convergent for $p>\frac{1}{2}$ entailing the continuity of the function $t\mapsto |A^{-p}\delta_t|_0^2$.\\
\noindent Now consider
\begin{eqnarray*}
W_t=\sum_{k\geq 0}\xi_k(t)I_1(\xi_k).
\end{eqnarray*}
Since 
\begin{eqnarray*}
\Gamma(A^{-p})W_t=\sum_{k\geq 0}\Big(k+\frac{3}{2}\Big)^{-p}\xi_k(t)I_1(\xi_k)
\end{eqnarray*}
and the $I_1(\xi_k)$'s are independent and identical distributed random variables, we deduce that for $p>1$ the last series is almost surely uniformly convergent and hence the process $\{\Gamma(A^{-p})W_t\}_{t\geq 0}$ is almost surely continuous.
\end{proof}

\noindent Observe that for $t\geq 0$ and $p>\frac{1}{2}$ we have
\begin{eqnarray*} 
\int_{0}^{t}\Vert W_{s}^{\diamond 2}\Vert_{-p}ds&=&\int_{0}^{t} \|(\Gamma(A^{-p})W_{s})^{\diamond 2}\|_{0}ds\\
&=&\int_{0}^{t} \|I_{1}(A^{-p}\delta_{s})^{\diamond 2}\|_{0}ds \\
&=&\int_{0}^{t} \|I_{2}((A^{-p}\delta_{s})^{\otimes 2})\|_{0}ds\\
&=&\int_{0}^{t}(2(|A^{-p}\delta_{s}|_{0}^{2})^{2})^{\frac{1}{2}}ds\\
&=&\sqrt{2}\int_{0}^{t}|A^{-p}\delta_{s}|_{0}^{2} ds\\
&=&\sqrt{2}\int_{0}^{t}\sum_{k\geq 0}\Big(k+\frac{3}{2}\Big)^{-2p}\xi_k(s)^{2}ds\\
&=&\sqrt{2}\sum_{n\geq 0}\Big(k+\frac{3}{2}\Big)^{-2p}\int_{0}^{t}\xi_k(s)^{2}ds\\
&\leq&\sqrt{2}\sum_{k\geq 0}\Big(k+\frac{3}{2}\Big)^{-2p}
\end{eqnarray*}
where the last series is convergent. 

\begin{definition}
For each $t\geq 0$ the process
\begin{eqnarray*}
t\mapsto X_t:=\int_0^tW_s^{\diamond 2}ds
\end{eqnarray*}
is an element of $(S)^*$. We will refer to it as the \emph{quadratic white noise process}.
\end{definition} 

Note that, since
\begin{eqnarray*}
\Gamma(A^{-p})X_t&=&\int_0^t(\Gamma(A^{-p})W_s)^{\diamond 2}ds\\
&=&\int_0^t(\Gamma(A^{-p})W_s)^2-|A^{-p}\delta_s|_0^2ds,
\end{eqnarray*}
we deduce from the previous lemma that the function $t\to\Gamma(A^{-p})X_t$ is almost surely differentiable for $p>1$.\\
\noindent The next theorem is the main result of the present section. Its proof makes use of several formulas relating the Malliavin derivative $D$, translation operator $T$ and Wick product $\diamond$; the books \cite{HOUZ}, \cite{J} and \cite{Kuo} are excellent references for the definition and properties of those operators as well as for the formulas just mentioned.  
\begin{theorem}
Let $\varphi\in C^{3}(\mathbb{R})$ be such that for each $i\in\{0,1,2,3\}$, $\varphi^{(i)}$ has at most polynomial growth at infinity. Then for any $p>1$ one has
\begin{eqnarray*}
\tilde{\varphi}_{p}(X_{t})&=&\varphi(0)+4\int_{0}^{t}\widetilde{\varphi_p'''}(X_s)\star_p \Big(\int_{0}^{s}K_{p}(r,s)dB_{r}\Big)^{\star_p 2}ds+2\int_0^t\widetilde{\varphi_p''}(X_s)\cdot\Big(\int_{0}^{s}K_{p}^{2}(r,s)dr\Big)ds\\
&&+4\int_{0}^{t}\Big(\widetilde{\varphi_p''}(X_{s})\star_p\int_{0}^{s}K_{p}(r,s)dB_{r}\Big)\diamond W_{s}ds+\int_0^t\widetilde{\varphi_p'}(X_{s})\diamond \frac{dX_s}{ds}ds,
\end{eqnarray*}
where $K_{p}(r,s):=\langle\delta_r^p,\delta_s^p\rangle$ and $\delta_u^p:=A^{-p}\delta_u$ for any $u\in\mathbb{R}$. Moreover, $(\int_{0}^{s}K_{p}(r,s)dB_{r})^{\star_p 2}$ is a shorthand notation for $\int_{0}^{s}K_{p}(r,s)dB_{r}\star_p\int_{0}^{s}K_{p}(r,s)dB_{r}$. 
\end{theorem}
\begin{proof}
We start applying the $S$-transform; for $f\in S$ one has
\begin{eqnarray*}
 \langle\langle \widetilde{\varphi_{p}}(X_{t}),\mathcal{E}(f)\rangle\rangle&=& \langle\langle
\Gamma(A^{p})\varphi(\Gamma(A^{-p})X_{t}),\mathcal{E}(f)\rangle\rangle\\
&=&E[\varphi(\Gamma(A^{-p})X_{t})\mathcal{E}(A^{p}f)] \\
&=&E[T_{A^{p}f}\varphi(\Gamma(A^{-p})X_{t})] \\
&=& E[\varphi(T_{A^{p}f}\Gamma(A^{-p})X_{t})]\\
&=&E[\varphi(\Gamma(A^{-p})T_{f}X_{t})].
\end{eqnarray*}
Here we utilized the Girsanov theorem and a commutation relation between the translation operator $T$ and the operator $\Gamma(A^{-p})$. Now observe that
\begin{eqnarray*}
T_{f}X_{t}=\int_{0}^{t}(W_{s}+f(s))^{\diamond2}ds=\int_{0}^{t}W^{\diamond2}_{s}ds+2\int_{0}^{t}W_{s}f(s)ds+\int_{0}^{t}f^{2}(s)ds
\end{eqnarray*}
and
\begin{eqnarray*}
\Gamma(A^{-p})T_{f}X_{t}=\int_{0}^{t}(\Gamma(A^{-p})W_{s})^{\diamond2}ds+2\int_{0}^{t}\Gamma(A^{-p})W_{s}\cdot f(s)ds+\int_{0}^{t}f^{2}(s)ds.
\end{eqnarray*}
To ease the notation, write $W_t^p$ for $\Gamma(A^{-p})W_t$; therefore
\begin{eqnarray*}
\langle\langle
\tilde{\varphi}_p(X_{t}),\mathcal{E}(f)\rangle\rangle&=&E[\varphi(\Gamma(A^{-p})T_{f}X_{t})]\\
&=&E\Big[\varphi\Big(\int_{0}^{t}(W^p_{s})^{\diamond2}ds+2\int_{0}^{t}W^p_{s}\cdot f(s)ds+\int_{0}^{t}f^{2}(s)ds\Big)\Big].
\end{eqnarray*}
To further facilitate the writing of the next calculation we also set
\begin{eqnarray*}
Y_{t}:=\int_{0}^{t}(W^p_{s})^{\diamond2}ds+2\int_{0}^{t}W^p_{s}\cdot
f(s)ds+\int_{0}^{t}f^{2}(s)ds.
\end{eqnarray*} 
We now apply the usual chain rule (the discussion preceding the statement of the theorem guarantees the differentiability of $Y_t$ while the assumptions on $\varphi$ allow us to differentiate inside the expected value) to obtain
\begin{eqnarray*}
\frac{d}{dt}\langle\langle\tilde{\varphi}_p(X_t),\mathcal{E}(f)\rangle\rangle&=&\frac{d}{dt}E[\varphi(Y_{t})]\\
&=&E[\varphi'(Y_{t})((W^p_{t})^{\diamond2}+2 W^p_{t}\cdot
f(t)+f^{2}(t))]\\
&=&E[\varphi'(Y_{t})(W^p_{t})^{\diamond2}]+2E[\varphi'(Y_{t})W^p_{t}]\cdot f(t)+E[\varphi'(Y_{t})]f^{2}(t)\\
&=&\mathcal{A}+\mathcal{B}+\mathcal{C},
\end{eqnarray*}
where
\begin{eqnarray*}
\mathcal{A}&:=&E[\varphi'(Y_{t})(W^p_{t})^{\diamond2}]\\
\mathcal{B}&:=&2E[\varphi'(Y_{t})W^p_{t}]\cdot f(t)\\
\mathcal{C}&:=&E[\varphi'(Y_t)]\cdot f^{2}(t).
\end{eqnarray*}
Recalling that $Y_t=\Gamma(A^{-p})T_f X_t$ we can write
\begin{eqnarray*}
\mathcal{A}&=&
E[\varphi'(Y_{t})(W^p_{t})^{\diamond2}]\\
&=&E[\varphi'(\Gamma(A^{-p})T_{f}X_{t})(W_{t}^{p})^{\diamond2}]\\
&=&E[T_{A^{p}f}(\varphi'(\Gamma(A^{-p})X_{t}))\cdot (W_{t}^{p})^{\diamond2}]\\
&=&E[\varphi'(\Gamma(A^{-p})X_{t})\cdot ((W_{t}^{p})^{\diamond2}\diamond\mathcal{E}(A^{p}f))].
\end{eqnarray*}
We now integrate twice by parts to obtain
\begin{eqnarray*}
\mathcal{A}&=&E[\varphi'(\Gamma(A^{-p})X_{t})\cdot ((W_{t}^{p})^{\diamond2}\diamond\mathcal{E}(A^{p}f))]\\
&=&E[D^{2}_{\delta_{t}^{p}}(\varphi'(\Gamma(A^{-p})X_{t}))\cdot\mathcal{E}(A^{p}f)]\\
&=&\langle\langle\Gamma(A^{p})D^{2}_{\delta_{t}^{p}}(\varphi'(\Gamma(A^{-p})X_{t})),\mathcal{E}(f)\rangle\rangle
\end{eqnarray*}
that means
\begin{eqnarray*}
\mathcal{A}=(S \Gamma(A^{p})D^{2}_{\delta_{t}^{p}}(\varphi'(\Gamma(A^{-p})X_{t})))(f).
\end{eqnarray*}
Moreover, by simple direct calculations we get
\begin{eqnarray*}
D^{2}_{\delta_{t}^{p}}(\varphi'(\Gamma(A^{-p})X_{t})) &=& D_{\delta_{t}^{p}}(D_{\delta_{t}^{p}}(\varphi'(\Gamma(A^{-p})X_{t}))) \\
&=&D_{\delta_{t}^{p}}(\varphi''(\Gamma(A^{-p})X_{t})\cdot D_{\delta_{t}^{p}}\Gamma(A^{-p})X_{t}) \\
&=& \varphi'''(\Gamma(A^{-p})X_{t})\cdot (D_{\delta_{t}^{p}}\Gamma(A^{-p})X_{t})^{2}+\varphi''(\Gamma(A^{-p})X_{t})\cdot  D^{2}_{\delta_{t}^{p}}\Gamma(A^{-p})X_{t},
\end{eqnarray*}
as well as
\begin{eqnarray*}
D_{\delta_{t}^{p}}\Gamma(A^{-p})X_{t}&=&D_{\delta_{t}^{p}}\Big(\int_{0}^{t}(W_{s}^{p})^{\diamond2}ds\Big) \\
&=& 2\int_{0}^{t}W_{s}^{p}\diamond D_{\delta_{t}^{p}}W_{s}^{p}ds\\
&=& 2\int_{0}^{t}W_{s}^{p}\cdot \langle \delta_{s}^{p},\delta_{t}^{p}\rangle ds\\
&=& 2\Gamma(A^{-p})\Big(\int_{0}^{t}K_{p}(s,t)dB_{s}\Big),
\end{eqnarray*}
and
\begin{eqnarray*}
 D^{2}_{\delta_{t}^{p}}\Gamma(A^{-p})X_{t} &=&D_{\delta_{t}^{p}}\Big(2\int_{0}^{t}W_{s}^{p}K_{p}(s,t)ds\Big)  \\
 &=& 2\int_{0}^{t}K_{p}^{2}(s,t)ds.
\end{eqnarray*}
Plugging all the preceding quantities together we deduce
\begin{eqnarray*}
\Gamma(A^{p})D^{2}_{\delta_{t}^{p}}(\varphi'(\Gamma(A^{-p})X_{t})) &=&\Gamma(A^{p})\Big(\varphi'''(\Gamma(A^{-p})X_{t})\cdot \Big(2\Gamma(A^{-p})\int_{0}^{t}K_{p}(s,t)dB_{s}\Big)^2\Big)\\
&&+\Gamma(A^{p})\Big(\varphi''(\Gamma(A^{-p})X_{t})\cdot 2\int_{0}^{t}K_{p}^{2}(s,t)ds\Big)\\
&=&4\widetilde{\varphi_p'''}(X_t)\star_{p}\Big(\int_{0}^{t}K_{p}(s,t)dB_{s}\Big)^{\star_p 2}+2\widetilde{\varphi_p''}(X_t)\cdot\int_{0}^{t}K_{p}^{2}(s,t)ds.
\end{eqnarray*}
Similarly for the second term,
\begin{eqnarray*}
\mathcal{B}&=&2E[\varphi'(Y_{t})W^p_{t}]\cdot f(t)\\
&=&2 E[T_{A^{p}f}(\varphi'(\Gamma(A^{-p})X_{t})) W_{t}^{p}]\cdot f(t)\\
&=&2 E[\varphi'(\Gamma(A^{-p})X_{t})\cdot
(W_{t}^{p}\diamond\mathcal{E}(A^{p}f)) ]\cdot f(t)\\
&=&2 E[D_{\delta_{t}^{p}}(\varphi'(\Gamma(A^{-p})X_{t}))\cdot\mathcal{E}(A^{p}f)
]\cdot f(t)\\
&=&2 \langle\langle\Gamma(A^{p})D_{\delta_{t}^{p}}(\varphi'(\Gamma(A^{-p})X_{t})),\mathcal{E}(f)\rangle\rangle\cdot f(t).
\end{eqnarray*}
That means
\begin{eqnarray*}
\mathcal{B}=2(S \Gamma(A^{p})D_{\delta_{t}^{p}}(\varphi'(\Gamma(A^{-p})X_{t}))\diamond
W_{t})(f);
\end{eqnarray*}
moreover
\begin{eqnarray*}
2\Gamma(A^{p})D_{\delta_{t}^{p}}(\varphi'(\Gamma(A^{-p})X_{t}))\diamond
W_{t}&=&2\Gamma(A^{p})(\varphi''(\Gamma(A^{-p})X_{t})D_{\delta_t^p}\Gamma(A^{-p})X_t)\diamond W_{t}\\
&=& 4\Gamma(A^p)\Big(\varphi''(\Gamma(A^{-p})X_t)\cdot\Gamma(A^{-p})\int_0^tK_p(s,t)dB_s\Big)\diamond W_t\\
&=& 4 \Big(\widetilde{\varphi'''}_{p}(X_{t})\star_p \Big(\int_{0}^{t}K_{p}(s,t)dB_{s}\Big)\Big)\diamond W_{t}.
\end{eqnarray*}
The third term is simply
\begin{eqnarray*}
\mathbf{C}=(S\tilde{\varphi}'_{p}(X_{t})\diamond W_{t}^{\diamond2})(f).
\end{eqnarray*}
Finally, we can write
\begin{eqnarray*}
\langle\langle\tilde{\varphi}_p(X_t),\mathcal{E}(f)\rangle\rangle&=&\varphi(0)+\int_0^t\frac{d}{ds}\langle\langle\tilde{\varphi}_p(X_s),\mathcal{E}(f)\rangle\rangle ds\\
&=&\varphi(0)+\int_0^t\frac{d}{ds}E[\varphi(Y_s)] ds\\
&=&\varphi(0)+\int_0^t\mathcal{A}+\mathcal{B}+\mathcal{C}  ds\\
&=&\varphi(0)+\int_0^t\Big\langle\Big\langle4\widetilde{\varphi_p'''}(X_s)\star_{p}\Big(\int_{0}^{s}K_{p}(r,s)dB_{r}\Big)^{\star_p 2},\mathcal{E}(f)\Big\rangle\Big\rangle ds\\
&&+\int_0^t\Big\langle\Big\langle2\widetilde{\varphi_p''}(X_s)\cdot\int_{0}^{s}K_{p}^{2}(r,s)dr,\mathcal{E}(f)\Big\rangle\Big\rangle ds\\
&&+\int_0^t\Big\langle\Big\langle4 \Big(\widetilde{\varphi_p'''}(X_{s})\star_p \Big(\int_{0}^{s}K_{p}(r,s)dB_{r}\Big)\Big)\diamond W_{s},\mathcal{E}(f)\Big\rangle\Big\rangle ds\\
&&+\int_0^t\langle\langle\tilde{\varphi}'_{p}(X_{s})\diamond W_{s}^{\diamond2},\mathcal{E}(f)\rangle\rangle ds.
\end{eqnarray*}
This completes the proof.
\end{proof}

\begin{example}
Choose $\varphi(x)=x^2$; then from the previous theorem we get 
\begin{eqnarray*}
X_t^{\star_p 2}&=&4\int_{0}^{t}\int_0^sK_{p}^{2}(r,s)dr ds+8\int_{0}^{t}\int_0^sK_{p}(r,s)dB_rdB_s+2\int_0^tX_s\diamond\frac{dX_s}{ds}ds \\
&=&4\int_{0}^{t}\int_0^sK_{p}^{2}(r,s)dr ds+8\int_{0}^{t}\int_0^sK_{p}(r,s)dB_rdB_s+X_t^{\diamond 2}\\
\end{eqnarray*}
or equivalently
\begin{eqnarray*}
X_t^{\star_p 2}-X_t^{\diamond 2}&=&4\int_{0}^{t}\int_0^sK_{p}^{2}(r,s)dr ds+8\int_{0}^{t}\int_0^sK_{p}(r,s)dB_rdB_s.
\end{eqnarray*}
Since the right hand side of the last formula is a smooth random variable, we deduce that the singular part of $X_t^{\star_p 2}$ coincides with $X_t^{\diamond 2}$. 
\end{example}

\section{Stochastic differential equations driven by the quadratic white noise process}

In this section we want to study stochastic differential equations driven by the quadratic white noise process. We will focus our attention on equations of the form
\begin{eqnarray}\label{SDE}
\frac{dY_t}{dt}&=&\tilde{b}_{p}(Y_{t})+Y_{t}\star_{p}W_{t}^{\diamond2},\quad t\geq0\quad\quad Y_0=x\in\mathbb{R}
\end{eqnarray}
where $p$ is a fixed positive real number, $x\in\mathbb{R}$ and $b:\mathbb{R}\to\mathbb{R}$ is a measurable function. First of all we define what we mean by solving equation (\ref{SDE}).
\begin{definition}\label{def solution} 
Given $T>0$ the process $\{Y_{t}\}_{0\leq t< T}$ is a
solution (up to time $T$) to equation (\ref{SDE}) if the following conditions are satisfied:
\begin{itemize}
\item  For any $t< T$ we have $Y_{t}\in (S_{-p})$
\item The processes 
\begin{eqnarray*}
s\longmapsto \tilde{b}_p(Y_{s}) \quad\mbox{ and }\quad  s\longmapsto Y_{s}\star_{p} W_{s}^{\diamond2}
\end{eqnarray*}
are Pettis integrable in $(S)^*$ over the interval $[0,t]$ for any $t<T$.
\item For any $\varphi\in (S)$ and $t\in [0,T[$ the following identity holds
\begin{eqnarray}\label{equation}
\langle\langle Y_{t},\varphi\rangle\rangle=\Big\langle\Big\langle x+\int_{0}^{t}\widetilde{b_{p}}(Y_{s})ds+\int_{0}^{t}Y_{s}\star_{p} W_{s}^{\diamond2} ds,\varphi\Big\rangle\Big\rangle
\end{eqnarray}
\item The process $\{Y_t\}_{0\leq t<T}$ is $\{\mathcal{F}_t\}_{0\leq t<T}$-adapted.
\end{itemize}
\end{definition}

\noindent The following is the main theorem of the present section.
\begin{theorem}
Let $p>1$ and $b:\mathbb{R}\to\mathbb{R}$ be globally Lipschitz continuous and with at most linear growth at infinity, i.e. 
\begin{eqnarray*}
|b(x)-b(y)|\leq C|x-y|\quad\mbox{ and }\quad|b(x)|\leq C(1+|x|)
\end{eqnarray*}
for some positive constant $C$ and all $x,y\in\mathbb{R}$. Then equation (\ref{SDE}) has a unique solution up to time $T=(4\sup_{r\in\mathbb{R}_+}|A^{-p}\delta_r|_0^2)^{-1}$.
\end{theorem}

\begin{proof}
We begin with the problem of existence of a solution. Let $\{V^p_{t}\}_{t\geq 0}$ be the unique solution of the following
random differential equation
\begin{eqnarray}\label{SDE for V}
\frac{dV^p_t}{dt}=b(V^p_{t}\exp\{\zeta^p_{t}\})\exp\{-\zeta^p_{t}\},\quad t\geq 0\quad\quad V^p_0=x\in\mathbb{R}
\end{eqnarray}
where we set $\zeta^p_s:=\int_0^s(W_r^p)^{\diamond 2}dr$ and $W_r^p:=\Gamma(A^{-p})W_r$. (Recall that by Lemma \ref{technical lemma} the function $s\mapsto\zeta_s^p$ is differentiable; this fact, together with the assumptions on $b$, guarantees the existence of a unique solution to equation (\ref{SDE for V})). Define 
\begin{eqnarray*}
Y_{t}:=\Gamma(A^{p})(\exp\{\zeta^p_{t}\}\cdot V^p_{t}); 
\end{eqnarray*}
we want to prove that $\{Y_t\}_{t\geq 0}$ is a solution to (\ref{SDE}) according to Definition \ref{def solution}.\\
\noindent First of all, by definition
$Y_{t}\in (S_{-p})$ if and only if $\exp\{\zeta^p_{t}\}V^p_{t}\in
(L^{2})$.  To prove this, observe that from (\ref{SDE for V}) we have
\begin{eqnarray*}
 |V^p_{t}| &\leq& |x|+ \int_{0}^{t}|b(V^p_{t}\exp\{\zeta^p_{t}\})|\exp\{-\zeta^p_{s}\}ds\\
  &\leq& |x|+\int_{0}^{t}C(1+|V^p_{s}|\exp\{\zeta^p_{s}\})\exp\{-\zeta^p_{s}\}ds \\
  &=&|x|+C\int_{0}^{t}\exp\{-\zeta^p_{s}\}ds+C\int_{0}^{t}|V^p_{s}|ds.
\end{eqnarray*}
Hence by the Gronwall inequality  we get
\begin{eqnarray}\label{gronwall}
|V^p_{t}|\leq h(t)+C\int_{0}^{t}h(s)e^{C(t-s)}ds
\end{eqnarray}
where we set, for notational convenience, $h(t):=|x|+C\int_{0}^{t}\exp\{-\zeta^p_{s}\}ds$. On the other hand, since
\begin{eqnarray*}
\zeta^p_{s}&=&\int_{0}^{s}(W_{r}^{p})^{\diamond2}dr\\
&=&\int_{0}^{s}(W_{r}^{p})^{2}-|\delta_{r}^{p}|_{0}^{2}dr
\end{eqnarray*}
we can bound as
\begin{eqnarray*}
\exp\{-\zeta^p_{s}\} &=& \exp\Big\{-\int_{0}^{s}(W_{r}^{p})^{2}dr+\int_{0}^{s}|\delta_{r}^{p}|_{0}^{2}dr\Big\}\\
&\leq& \exp\Big\{\int_{0}^{s}|\delta_{r}^{p}|_{0}^{2}dr\Big\}.
\end{eqnarray*}
This last estimate yields that
\begin{eqnarray*}
h(t)\leq |x|+C\int_0^t\exp\Big\{\int_0^s|\delta_r^p|_0^2dr\Big\}ds
\end{eqnarray*}
and hence from inequality (\ref{gronwall}) the boundedness of
$V^p_{t}$ as a function of $\omega\in\Omega$. Therefore, $Y_t\in (S_{-p})$ if and only if $\exp\{\zeta^p_t\}\in (L^2)$. As it is stated in the following lemma, this is true for small enough $t$.
\begin{lemma}\label{jensen}
Let $q\geq 1$. If $t<(2q\sup_{r\in\mathbb{R}_+}|\delta_{r}^{p}|_{0}^{2})^{-1}$, then $\exp\{\zeta^p_{t}\}\in(L^q)$.
\end{lemma}
\begin{proof}
We use the Jensen inequality for the exponential function and the normalized Lebesgue measure: 
\begin{eqnarray*}
E[|\exp\{\zeta_t\}|^q]&=&E\Big[\exp\Big\{q\int_0^t(W_r^p)^2dr-q\int_0^t|\delta_r^p|_0^2dr\Big\}\Big]\\
&\leq&E\Big[\exp\Big\{q\int_{0}^{t}(W_{r}^{p})^{2}dr\Big\}\Big]\\
&=&E\Big[\exp\Big\{qt\cdot\frac{1}{t}\int_{0}^{t}(W_{r}^{p})^{2}dr\Big\}\Big] \\
&\leq&E\Big[\frac{1}{t}\int_{0}^{t}\exp\{qt\cdot
(W_{r}^{p})^{2}\}dr\Big]\\
&=&\frac{1}{t}\int_{0}^{t}E[\exp\{qt\cdot
(W_{r}^{p})^{2}\}]dr \\
&=&\frac{1}{t}\int_{0}^{t}\int_{\mathbb{R}}\exp\{qty^{2}\}\frac{1}{\sqrt{2\pi |\delta_r^p|_0^2}}e^{-\frac{y^{2}}{2 |\delta_r^p|_0^2}}dydr\\
&=&\frac{1}{t}\int_{0}^{t}\int_{\mathbb{R}}\frac{1}{\sqrt{2\pi |\delta_r^p|_0^2}}\exp\Big\{-\frac{y^{2}}{2}\Big(\frac{1}{|\delta_r^p|_0^2}-2qt\Big)\Big\}dydr
\end{eqnarray*}
and the last inner integral is finite if  $t<\frac{1}{2q|\delta_{r}^{p}|_{0}^{2}}$ for all $r\in [0,t]$. The condition $t<(2q\sup_{r\in\mathbb{R}_+}|\delta_{r}^{p}|_{0}^{2})^{-1}$ is therefore sufficient for the inner integral to be finite and the continuity of $r\mapsto|\delta_r^p|_0^2$ implies also the finiteness of the other integral. 
\end{proof}

\noindent Therefore with the help of the previous lemma we deduce that $Y_t\in (S_{-p})$ if $t<T:=(4\sup_{r\in\mathbb{R}_+}|\delta_{r}^{p}|_{0}^{2})^{-1}$.\\
We now prove that for any $t<T$, the process $s\mapsto\tilde{b}_p(Y_s)$ is Pettis integrable over the interval $[0,t]$. Observe that
\begin{eqnarray*}
\tilde{b}_p(Y_s)&=&\Gamma(A^p)b(\Gamma(A^{-p})Y_s)\\
&=&\Gamma(A^p)b(V^p_s\exp\{\zeta^p_s\}).
\end{eqnarray*}
Therefore for any $u>0$, $z\in\mathbb{C}$ and $f\in S$,
\begin{eqnarray*}
\int_{0}^{t}|(S \Gamma(A^p)b(V^p_{s}\exp\{\zeta^p_{t}\}))(zf)|ds&=&\int_{0}^{t}|(S b(V^p_{s}\exp\{\zeta^p_{t}\}))(zA^{p}f)|ds\\ 
&\leq& \int_{0}^{t}\|b(V^p_{s}\exp\{\zeta^p_{s}\})\|_{-u}e^{\frac{|z|^{2}}{2}|A^{p+u}f|_{0}^{2}}ds \\
&\leq& e^{\frac{|z|^{2}}{2}|A^{p+u}f|_{0}^{2}} \int_{0}^{t}\|b(V^p_{s}\exp\{\zeta^p_{s}\})\|_{0}ds \\
&\leq& Ce^{\frac{|z|^{2}}{2}|A^{p+u}f|_{0}^{2}} \int_{0}^{t}1+\|V^p_{s}\exp\{\zeta^p_{s}\}\|_{0}ds\\  
&\leq&Ce^{\frac{|z|^{2}}{2}|A^{p+u}f|_{0}^{2}}\Big(t+\sup_{s\in[0,t]}\|V^p_{s}\|_{(L^{\infty})}\int_{0}^{t}\|\exp\{\zeta^p_{s}\}\|_{0}ds\Big).
\end{eqnarray*}
From (\ref{gronwall}) we get that $\sup_{s\in[0,t]}\|V^p_{s}\|_{(L^{\infty})}$ is finite; for the integral we can write
\begin{eqnarray*}
\int_{0}^{t}\|\exp\{\zeta^p_{s}\}\|_{0}ds&=&\int_0^tE[\exp\{2\zeta^p_s\}]^{\frac{1}{2}}ds\\
&\leq&\int_0^tE\Big[\exp\Big\{2\int_0^s(W_r^p)^2dr\Big\}\Big]^{\frac{1}{2}}ds\\
&\leq&tE\Big[\exp\Big\{2\int_0^t(W_r^p)^2dr\Big\}\Big]^{\frac{1}{2}}
\end{eqnarray*}
and the last expected value is finite by Lemma \ref{jensen}.\\
Now we prove that the process $s\mapsto Y_{s}\star_{p}W_{s}^{\diamond2}$ is Pettis integrable. By definition
\begin{eqnarray*}
Y_{s}\star_{p}W_{s}^{\diamond2}&=&\Gamma(A^{p})(\Gamma(A^{-p})Y_{s}\cdot\Gamma(A^{-p})W_{s}^{\diamond 2})\\
&=& \Gamma(A^{p})(V^p_s\exp\{\zeta^p_{t}\}\cdot (W^p_{s})^{\diamond 2}).
\end{eqnarray*}
Hence for any $u>0$, $z\in\mathbb{C}$ and $f\in S$,
\begin{eqnarray*}
\int_{0}^{t}|(S \Gamma(A^p)(\exp\{\zeta^p_{s}\}V^p_{s}\cdot (W^p_{s})^{\diamond 2}))(zf)|ds&=&\int_{0}^{t}|(S \exp\{\zeta^p_{s}\}V^p_{s}\cdot (W^p_{s})^{\diamond 2})(zA^{p}f)|ds\\  
&\leq& \int_{0}^{t}\|\exp\{\zeta^p_{s}\}V^p_{s}\cdot (W^p_{s})^{\diamond 2}\|_{-u}e^{\frac{|z|^{2}}{2}|A^{p+u}f|_{0}^{2}}ds\\
&\leq& e^{\frac{|z|^{2}}{2}|A^{p+u}f|_{0}^{2}}\sup_{s\in [0,t]}\Vert V^p_s\Vert_{(L^{\infty})}\int_{0}^{t}\|\exp\{\zeta^p_{s}\}\cdot (W^p_{s})^{\diamond2}\|_{0}ds\\
&\leq& e^{\frac{|z|^{2}}{2}|A^{p+u}f|_{0}^{2}}\sup_{s\in [0,t]}\Vert V^p_s\Vert_{(L^{\infty})}\\
&&\cdot\int_{0}^{t}\|\exp\{\zeta^p_{s}\}\|_{(L^{2+\varepsilon})}\cdot\|(W^p_{s})^{\diamond2}\|_{(L^q)}ds
\end{eqnarray*}
where $\varepsilon>0$ and $q>2$ is such that $\frac{1}{2}=\frac{1}{2+\varepsilon}+\frac{1}{q}$ (here we utilized the H\"older inequality). Since $(W^p_{s})^{\diamond2}$ has a finite chaos expansion, by the Nelson's hyper-contractive estimate (\cite{Nelson}) we can bound its norm as
\begin{eqnarray*}
\|(W^p_{s})^{\diamond2}\|_{(L^q)}\leq K\|(W^p_{s})^{\diamond2}\|_{(L^2)}
\end{eqnarray*}
where the constant $K$ depends only on $q$; moreover $\|(W^p_{s})^{\diamond2}\|_{(L^2)}$ is a continuous function of $s$ (since $p>1$). The term $\|\exp\{\zeta^p_{s}\}\|_{(L^{2+\varepsilon})}$ can be treated as before with the help of Lemma \ref{jensen}. All these facts provide the finiteness of the quantity
\begin{eqnarray*}
\int_{0}^{t}\|\exp\{\zeta^p_{s}\}\|_{(L^{2+\varepsilon})}\cdot\|(W^p_{s})^{\diamond2}\|_{(L^q)}ds
\end{eqnarray*} 
and hence the bound 
\begin{eqnarray*}
\int_{0}^{t}|(S \Gamma(A^p)(\exp\{\zeta^p_{s}\}V^p_{s}\cdot (W^p_{s})^{\diamond 2}))(zf)|ds\leq Ce^{\frac{|z|^2}{2}|A^{p+u}f|_0^2}.
\end{eqnarray*}
 
\noindent We now verify equation (\ref{equation}). For all $\varphi\in (S)$ and $t< T$ we have
\begin{eqnarray*}
\langle\langle Y_{t},\varphi\rangle\rangle&=&\langle\langle \Gamma(A^p)(V_t^p\exp\{\zeta_t^p\}),\varphi\rangle\rangle\\
&=&\langle\langle V_t^p\exp\{\zeta_t^p\},\Gamma(A^p)\varphi\rangle\rangle\\
&=&\Big\langle\Big\langle x+\int_0^t\frac{d}{ds}(V_s^p\exp\{\zeta_s^p\})ds,\Gamma(A^p)\varphi\Big\rangle\Big\rangle\\
&=&\Big\langle\Big\langle x+\int_0^t b(V_s^p\exp\{\zeta_s^p\})+V_s^p\cdot (W_s^p)^{\diamond 2}\exp\{\zeta_s^p\}ds,\Gamma(A^p)\varphi\Big\rangle\Big\rangle\\
&=&\Big\langle\Big\langle x+\int_0^t b(\Gamma(A^{-p})Y_s)+\Gamma(A^{-p})Y_s\cdot (W_s^p)^{\diamond 2}ds,\Gamma(A^p)\varphi\Big\rangle\Big\rangle\\
&=&\Big\langle\Big\langle x+\int_{0}^{t}\widetilde{b_{p}}(Y_{s})ds+\int_{0}^{t}Y_{s}\star_{p}
W_{s}^{\diamond2} ds,\varphi\Big\rangle\Big\rangle.
\end{eqnarray*}
(Here we utilized equation (\ref{SDE for V}); moreover the interchange between the unbounded operator $\Gamma(A^p)$ and the integral is allowed by the Pettis integrability proved before)\\

\noindent To conclude the proof we need to check the adaptedness of the process $\{Y_t\}_{t\in [0,T[}$. To do that we compute the $S$-transform of the solution at time $t\in [0,T[$, i.e. $(S Y_t)(f)$, and show that $(S Y_t)(f+g)=(S Y_t)(f)$ for all $f,g\in S$ such that the support of $g$ in contained in $[0,t]^c$. We have
\begin{eqnarray}\label{adapted}
(S Y_t)(f)&=&\langle\langle Y_t,\mathcal{E}(f)\rangle\rangle\nonumber\\
&=&\langle\langle\Gamma(A^p)(V^p_t\exp\{\zeta^p_t\}),\mathcal{E}(f)\rangle\rangle\nonumber\\
&=&\langle\langle V^p_t\exp\{\zeta^p_t\},\mathcal{E}(A^pf)\rangle\rangle\nonumber\\
&=&E[V^p_t\exp\{\zeta^p_t\}\cdot\mathcal{E}(A^pf)]\nonumber\\
&=&E[T_{A^pf}(V^p_t\exp\{\zeta^p_t\})]\nonumber\\
&=&E[T_{A^pf}V^p_t\cdot\exp\{T_{A^pf}\zeta^p_t\})]\nonumber\\
&=&E\Big[T_{A^pf}V^p_t\cdot\exp\Big\{\int_{0}^{t}(W^p_{s})^{\diamond 2}ds+2\int_{0}^{t}W^p_{s}\cdot
f(s)ds+\int_{0}^{t}f^{2}(s)ds\Big\}\Big]
\end{eqnarray}
where in the last equality we utilized the following
\begin{eqnarray*}
T_{A^pf}W_s^p&=&T_{A^pf}\int_{\mathbb{R}_+}(A^{-p}\delta_s)(r)dB_r\\
&=&\int_{\mathbb{R}_+}(A^{-p}\delta_s)(r)(dB_r+(A^pf)(r)dr)\\
&=&\int_{\mathbb{R}_+}(A^{-p}\delta_s)(r)dB_r+\int_{\mathbb{R}_+}(A^{-p}\delta_s)(r)(A^pf)(r)dr\\
&=&W_s^p+f(s).
\end{eqnarray*}
It is clear that the exponential appearing in (\ref{adapted}) remains unchanged if add to $f$ a function $g$ that is identically zero on the interval $[0,t]$. Let us see if the same is true for the term $T_{A^pf}V^p_t$ appearing in (\ref{adapted}); recall that $\{V^p_t\}_{t\in [0,T[}$ solves
\begin{eqnarray*}
V^p_t=x+\int_0^tb(V^p_s\exp\{\zeta^p_s\})\exp\{-\zeta^p_s\}ds.
\end{eqnarray*}
Apply $T_{A^pf}$ to both sides of the equation above to get
\begin{eqnarray}\label{translated V}
T_{A^pf}V^p_t=x+\int_0^tb(T_{A^pf}V^p_s\cdot T_{A^pf}\exp\{\zeta^p_s\})T_{A^pf}\exp\{-\zeta^p_s\}ds;
\end{eqnarray}
if we now replace in equation (\ref{translated V}) the function $f$ with $f+g$ where the support of $g$ is contained in $[0,t]^c$, then by the above mentioned invariance of the exponential appearing in that equation we obtain
\begin{eqnarray}\label{translated V 2}
T_{A^p(f+g)}V^p_t=x+\int_0^tb(T_{A^p(f+g)}V^p_s\cdot T_{A^pf}\exp\{\zeta^p_s\})T_{A^pf}\exp\{-\zeta^p_s\}ds.
\end{eqnarray}
From (\ref{translated V}) and (\ref{translated V 2}) we deduce that $T_{A^pf}V^p_t$ and $T_{A^p (f+g)}V^p_t$ solve the same equation. From uniqueness of the solution we deduce that $T_{A^pf}V^p_t=T_{A^p (f+g)}V^p_t$ and hence that $(S Y_t)(f+g)=(S Y_t)(f)$.\\

\noindent We now prove uniqueness. Assume that $\{Y_{t}\}_{t\in [0,T[}$ is a solution to equation (\ref{SDE}). Then applying the bounded operator $\Gamma(A^{-p})$ we get 
\begin{eqnarray*}
\langle\langle\Gamma(A^{-p})Y_{t},\varphi\rangle\rangle&=&\langle\langle Y_{t},\Gamma(A^{-p})\varphi\rangle\rangle\\
&=&\Big\langle\Big\langle x+\int_{0}^{t}\tilde{b}_p(Y_{s})ds+\int_{0}^{t}Y_{s}\star_p W_{s}^{\diamond2}ds,\Gamma(A^{-p})\varphi\Big\rangle\Big\rangle\\
&=&\Big\langle\Big\langle x+\int_{0}^{t}b(\Gamma(A^{-p})Y_{s})ds+\int_{0}^{t}\Gamma(A^{-p})Y_{s}\cdot\Gamma(A^{-p})W_{s}^{\diamond2}ds,\varphi\rangle\rangle
\end{eqnarray*}
or equivalently
\begin{eqnarray}\label{SDE for Z}
Z_{t}=x+\int_{0}^{t}b(Z_{s})ds+\int_{0}^{t}Z_{s}\cdot\Gamma(A^{-p})W_{s}^{\diamond2}ds\quad\mbox{ in }(S)^*
\end{eqnarray}
where we set $Z_{t}:=\Gamma(A^{-p})Y_{t}$. Observe that equation (\ref{SDE for Z}) is a random differential equation with a unique solution and therefore any other solution to (\ref{SDE}) is mapped by $\Gamma(A^{-p})$ to the same $Z_t$. This fact, together with the injectivity of $\Gamma(A^{-p})$, implies uniqueness of the solution to (\ref{SDE}).
\end{proof}

\begin{remark}
One may wonder whether a similar approach can be utilized to solve equations of the form
\begin{eqnarray*}
Y_t&=&x+\int_0^t\tilde{b}_{p}(Y_{s})ds+\int_0^tY_{s}\star_{p}U_sds,\quad t\geq0
\end{eqnarray*}
where $\{U_s\}_{s\geq 0}$ belongs to a more general class of processes taking values in the space $(S)^*$. It is clear from the proof of the previous theorem that only minor generalizations can be considered in this framework; more precisely, when we prove existence of a solution  in the space $(S)^*$, we have to deal with the problem of checking if the exponential of our smoothed driving noise is square integrable (see Lemma \ref{jensen}). If this noise would have non  zero components in Wiener chaoses of order greater than two, then the square integrability of that exponential would simply fail to be true. From this point of view our driving noise is already a border line case since the solution exists locally in time (due to the restriction imposed by the square integrability requirement).
\end{remark}

\begin{remark}
The quantity $(4\sup_{r\in\mathbb{R}_+} |\delta_r^p|_0^2)^{-1}$, which is an upper bound of the life time of the solution, is increasing with $p$, which is the degree of regularization introduced to define the non linearity $b$. In fact, if $p$ is big then $|\delta_r^p|_0^2$ is small and hence we can solve the equation up to late times.  
\end{remark}

\begin{remark}\label{approximation idea}
Looking through the proof of the previous theorem one can see that the solution $\{Y_t\}_{0\leq t<T}$ to equation (\ref{SDE}) is obtained as $Y_t=\Gamma(A^p)Z_t$ where $\{Z_t\}_{0\leq t<T}$ is the unique solution to 
\begin{eqnarray}\label{regular noise}
\frac{dZ_{t}}{dt}=b(Z_{t})+Z_{t}\cdot\Gamma(A^{-p})W_{t}^{\diamond2}.
\end{eqnarray}
The last equation corresponds to the smoothed version (regularization of the noise) of
\begin{eqnarray}\label{last equation}
\frac{d\hat{Z}_{t}}{dt}=b(\hat{Z}_{t})+\hat{Z}_{t}\cdot W_{t}^{\diamond2}.
\end{eqnarray}
This explains that the procedure utilized in this paper to renormalize and solve equation (\ref{last equation}) amounts at: smoothing the noise in (\ref{last equation}) by applying $\Gamma(A^{-p})$, solving the regularized equation (\ref{regular noise}) and removing the regularization by applying $\Gamma(A^p)$ to the solution to (\ref{regular noise}). 
\end{remark}

\begin{example}
We now want to study in some detail the linear case. Consider the equation
\begin{eqnarray*}
Y_t=1+\int_0^tY_sds+\int_0^tY_s\star_pW_s^{\diamond 2}ds.
\end{eqnarray*}
The unique solution to this equation is given by
\begin{eqnarray*}
Y_t&=&\Gamma(A^p)\exp\Big\{t+\int_0^t(W_s^p)^{\diamond 2}ds\Big\}\\
&=&\Gamma(A^p)\exp\Big\{t+\int_0^t(W_s^p)^2-|\delta_s^p|_0^2ds\Big\}\\
&=&\exp\Big\{t-\int_0^t|\delta_s^p|_0^2ds\Big\}\Gamma(A^p)\exp\Big\{\int_0^t(W_s^p)^{2}ds\Big\}\\
\end{eqnarray*}
where we utilized as before the notation $W_s^p:=\Gamma(A^{-p})W_s$. We want to prove that for any $t\in[0,T[$ the quantity $Y_t$ is a positive Hida distribution. Observe that the unboundedness of the operator $\Gamma(A^p)$ does not guarantee in general the preservation of the positivity of the exponential which is applied to. To prove the desired property we need to show that the function
\begin{eqnarray}\label{PD}
f\in S\mapsto (SY_t)(if)\exp\Big\{-\frac{|f|_0^2}{2}\Big\}
\end{eqnarray}
is positive definite, i.e. for any $z_1,..,z_n\in\mathbb{C}$ and $f_1,...,f_n\in S$ the following inequality must be true:
\begin{eqnarray*}
\sum_{j,l=1}^nz_j (SY_t)(if_j-if_l)\exp\Big\{-\frac{|if_j-if_l|_0^2}{2}\Big\}\bar{z_l}\geq 0
\end{eqnarray*}
where $i$ denotes the imaginary unit and $\bar{\cdot}$ stands for complex conjugation. First of all, observe that
\begin{eqnarray*}
(SY_t)(f)&=&E\Big[\exp\Big\{t-\int_0^t|\delta_s^p|_0^2ds\Big\}\exp\Big\{\int_0^t(W_s^p+f(s))^{2}ds\Big\}\Big]\\
&=&\exp\Big\{t-\int_0^t|\delta_s^p|_0^2ds+\int_0^tf^2(s)ds\Big\}E\Big[\exp\Big\{\int_0^t(W_s^p)^2ds+\int_0^tW_s^pf(s)ds\Big\}\Big]
\end{eqnarray*}
and hence that
\begin{eqnarray*}
(SY_t)(if)\exp\Big\{-\frac{|f|_0^2}{2}\Big\}&=&\exp\Big\{t-\int_0^t|\delta_s^p|_0^2ds-\int_0^tf^2(s)ds-\frac{|f|_0^2}{2}\Big\}\\
&&\cdot E\Big[\exp\Big\{\int_0^t(W_s^p)^2ds+i\int_0^tW_s^pf(s)ds\Big\}\Big].
\end{eqnarray*}
Recall that a product of positive definite functions is positive definite; moreover,
\begin{eqnarray*}
\exp\Big\{-\int_0^tf^2(s)ds\Big\}=E\Big[\exp\Big\{i\int_0^tf(s)dB_s\Big\}\Big]
\end{eqnarray*}
and
\begin{eqnarray*}
\exp\Big\{-\int_0^{+\infty}f^2(s)ds\Big\}=E\Big[\exp\Big\{i\int_0^{+\infty}f(s)dB_s\Big\}\Big]
\end{eqnarray*} 
hence they are both positive definite. The term $\exp\{t-\int_0^t|\delta_s^p|_0^2ds\}$ is positive and it does not depend on $f$. Therefore, to prove the positive definiteness of (\ref{PD}) it remains to verify that condition for 
\begin{eqnarray*}
E\Big[\exp\Big\{\int_0^t(W_s^p)^2ds+i\int_0^tW_s^pf(s)ds\Big\}\Big].
\end{eqnarray*}
But this is easily done; in fact,
\begin{eqnarray*}
&&\sum_{j,l=1}^nz_jE\Big[\exp\Big\{\int_0^t(W_s^p)^2ds+i\int_0^tW_s^p(f_j(s)-f_l(s))ds\Big\}\Big]\bar{z_l}\\
&=&\sum_{j,l=1}^nz_jE\Big[\exp\Big\{\int_0^t(W_s^p)^2ds\Big\}\exp\Big\{i\int_0^tW_s^pf_j(s)ds\Big\}\exp\Big\{-i\int_0^tW_s^pf_l(s)ds\Big\}\Big]\bar{z_l}\\
&=&E\Big[\exp\Big\{\int_0^t(W_s^p)^2ds\Big\}\Big|\sum_{j=1}^nz_j\exp\Big\{i\int_0^tW_s^pf_j(s)ds\Big\}\Big|^2\Big]\\
&\geq&0.
\end{eqnarray*}
One can prove in the same way that the process $\{Y_t\}_{0\leq t< T}$ is actually strongly positive, a notion introduced in \cite{NZ} (more stringent than positivity for Hida distributions) to treat problems about positivity of Wick products and related to measures of convolution type and Poincar\'e inequalities (see \cite{L} and \cite{DLS2}).
\end{example}

\section{Renormalization for distributions from the first Wiener chaos}
In this section we are going to investigate the properties of $\tilde{\varphi_p}(X)$ in the particular case where $X$ belongs for the first Wiener chaos,
 i.e. when $X$ can be written as $I_1(h_1)$ for some $h_1\in S_{-p}$ and $p\geq 0$.  To this aim we recall that for a given real analytic function $\psi:\mathbb{R}\to\mathbb{R}$, $\psi(x)=\sum_{n\geq 0}a_nx^n$ and an element $X\in (S)^*$ we can define 
\begin{eqnarray*}
\psi^{\diamond}(X):=\sum_{n\geq 0}a_nX^{\diamond n}
\end{eqnarray*}
where $X^{\diamond n}:=X\diamond\cdot\cdot\cdot\diamond X$ ($n$-times), provided the above series converges in $(S)^*$. If for instance we take $\psi(x)=\exp\{x\}$ and $X=I_1(h_1)$ for some $h_1\in\mathcal{L}^2(\mathbb{R}_+)$ then we get $\exp^{\diamond}\{I_1(h_1)\}=\mathcal{E}(h_1)$.\\
Let $\{P_t\}_{t\geq 0}$ denotes the one dimensional  heat semigroup, i.e. for $t\geq 0$ and $x\in\mathbb{R}$,
\begin{eqnarray*}
(P_t\varphi)(x):=\int_{\mathbb{R}}\varphi(y)\frac{1}{\sqrt{2\pi t}}\exp\Big\{-\frac{(x-y)^2}{2t}\Big\}dy.
\end{eqnarray*}
Then for any bounded $\varphi:\mathbb{R}\to\mathbb{R}$ and $t>0$ the function $x\mapsto (P_t\varphi)(x)$ is real analytic. Moreover, 
\begin{eqnarray}\label{kuo}
E[\varphi(X(t))^2]=\sum_{n\geq 0}\frac{t^n}{n!}((P_t\varphi)^{(n)}(0))^2
\end{eqnarray}
where $X(t)$ is a Gaussian random variable with mean zero and variance $t$ (see \cite{AKK}).

\begin{proposition}
Let  $h_1\in S_{-p}$ for some $p\geq 0$; then for any bounded function $\varphi:\mathbb{R}\to\mathbb{R}$ one has
\begin{eqnarray*}
\tilde{\varphi}_p(I_1(h_1))=(P_{|A^{-p}h_1|_0^2}\varphi)^{\diamond}(I_1(h_1)).
\end{eqnarray*}
\end{proposition}

\begin{proof}
By definition 
\begin{eqnarray*}
\tilde{\varphi}_p(I_1(h_1))&=&\Gamma(A^{p})(\varphi(\Gamma(A^{-p})I_1(h_1)))\\
&=&\Gamma(A^{p})(\varphi(I_1(A^{-p}h_1))).
\end{eqnarray*}
We now apply the $S$-transform to obtain
\begin{eqnarray*}
(S\tilde{\varphi}_p(I_1(h_1)))(f)&=&(S\Gamma(A^{p})(\varphi(I_1(A^{-p}h_1))) )(f)\\
&=&(S\varphi(I_1(A^{-p}h_1))) (A^pf)\\
&=&E[\varphi(I_1(A^{-p}h_1))\mathcal{E}(A^pf)]\\
&=&E[\varphi(I_1(A^{-p}h_1)+\langle A^{-p}h_1,A^pf\rangle)]\\
&=&E[\varphi(I_1(A^{-p}h_1)+\langle h_1,f\rangle)]
\end{eqnarray*}
where in the fourth equality we applied the Girsanov theorem. Observe that we can write the last term as
\begin{eqnarray*}
E[\varphi(I_1(A^{-p}h_1)+\langle h_1,f\rangle)]=(P_{|A^{-p}h_1|}\varphi)(\langle h_1,f\rangle).
\end{eqnarray*}
On the other hand, since the function $x\mapsto (P_t\varphi)(x)$ is analytic for any $t>0$ (by the discussion above), one has for any $f\in S$ the identity
\begin{eqnarray*}
(P_{|A^{-p}h_1|}\varphi)(\langle h_1,f\rangle)=\Big(S (P_{|A^{-p}h_1|^2}\varphi)^{\diamond}(I_1(h_1)) \Big)(f).
\end{eqnarray*}
We can then conclude that 
\begin{eqnarray*}
(S\tilde{\varphi}_p(I_1(h_1)))(f)=\Big(S (P_{|A^{-p}h_1|^2}\varphi)^{\diamond}(I_1(h_1)) \Big)(f)
\end{eqnarray*}
for any $f\in S$. The injectivity of the $S$-transform completes the proof.
\end{proof}

\begin{remark}
Observe that, if in the previous proposition we assume $h_1\in\mathcal{L}^2(\mathbb{R}_+)$, then we can choose $p=0$ and obtain
\begin{eqnarray*}
\varphi(I_1(h_1))=(P_{|h_1|^2}\varphi)^{\diamond}(I_1(h_1)).
\end{eqnarray*}
This equality has been already proved in \cite{L2}.
\end{remark}

It is interesting to measure how much $\tilde{\varphi}_p(X)$ differs from $\varphi(X)$ when there is no need of renormalization in the nonlinear function $\varphi$, i.e. when $X$ belongs to $(L^2)$. 
The next theorem provides a result in this direction for $X$ being an element in the first Wiener chaos.

\begin{theorem}
Assume that $\varphi:\mathbb{R}\to\mathbb{R}$ is twice continuously differentiable with bounded derivatives, $h_1\in
\mathcal{L}^{2}(\mathbb{R}_+)$ and $p\geq 0$. Then
\begin{equation}\label{bilel}
\|\varphi(I_1(h_1))-\tilde{\varphi}_p(I_1(h_1))\|_{-p}\leq
C\sup_{x\in\mathbb{R}}|\varphi''(x)|\cdot\frac{|h_1|_{0}^{2}-|A^{-p}h_1|_{0}^{2}}{2}
\end{equation}
where $C$ is a positive constant depending only on $p$ and $h_1$.
\end{theorem}
\begin{proof}
According to the previous proposition and remark we can write
\begin{eqnarray*}
\varphi(I_1(h_1))-\tilde{\varphi}_p(I_1(h_1))&=& (P_{|h_1|_{0}^{2}}\varphi)^{\diamond}(I_1(h_1))-(P_{|A^{-p}h_1|_{0}^{2}}\varphi)^{\diamond}(I_1(h_1))\\
&=& \sum_{n\geq 0}\frac{(P_{|h_1|_{0}^{2}}\varphi)^{(n)}(0)-(P_{|A^{-p}h_1|_{0}^{2}}\varphi)^{(n)}(0)}{n!}I_{1}(h_1)^{\diamond n}\\
&=& \sum_{n\geq 0}\frac{(P_{|h_1|_{0}^{2}}\varphi)^{(n)}(0)-(P_{|A^{-p}h_1|_{0}^{2}}\varphi)^{(n)}(0)}{n!}I_{n}(h_1^{\otimes n}).
 \end{eqnarray*}
Now, according to Lagrange Theorem there exists $\tau\in ]|A^{-p}h_1|_{0}^{2},|h_1|_{0}^{2}[$ such that
\begin{eqnarray*}
(P_{|h_1|_{0}^{2}}\varphi)^{(n)}(0)-(P_{|A^{-p}h_1|_{0}^{2}}\varphi)^{(n)}(0)=\frac{1}{2}(P_{\tau}\varphi'')^{(n)}(0)(|h_1|_{0}^{2}-|A^{-p}h_1|_{0}^{2});
\end{eqnarray*}
hence
\begin{eqnarray*}
\varphi(I_1(h_1))-\tilde{\varphi}_p(I_1(h_1))&=&\sum_{n\geq 0}\frac{(P_{|h_1|_{0}^{2}}\varphi)^{(n)}(0)-(P_{|A^{-p}h_1|_{0}^{2}}\varphi)^{(n)}(0)}{n!}I_{n}(h_1^{\otimes
 n})\\
&=&\frac{(|h_1|_{0}^{2}-|A^{-p}h_1|_{0}^{2})}{2}\sum_{n\geq 0}\frac{(P_{\tau}\varphi'')^{(n)}(0)}{n!}I_{n}(h_1^{\otimes n}).
\end{eqnarray*}
We now compute the $\Vert\cdot\Vert_{-p}$-norm to get
\begin{eqnarray*}
\Vert\varphi(I_1(h_1))-\tilde{\varphi}_p(I_1(h_1))\Vert_{-p}&=&\Big\Vert\frac{(|h_1|_{0}^{2}-|A^{-p}h_1|_{0}^{2})}{2}\sum_{n\geq 0}\frac{(P_{\tau}\varphi'')^{(n)}(0)}{n!}I_{n}(h_1^{\otimes n})\Big\Vert_{-p}\\
&\leq&\frac{(|h_1|_{0}^{2}-|A^{-p}h_1|_{0}^{2})}{2}\sum_{n\geq 0}\frac{|(P_{\tau}\varphi'')^{(n)}(0)|}{n!}\Vert I_{n}(h_1^{\otimes n})\Vert_{-p}\\
&=&\frac{(|h_1|_{0}^{2}-|A^{-p}h_1|_{0}^{2})}{2}\sum_{n\geq 0}\frac{|(P_{\tau}\varphi'')^{(n)}(0)|}{n!}\sqrt{n!}|A^{-p}h_1|_0^{n}\\
&=&\frac{(|h_1|_{0}^{2}-|A^{-p}h_1|_{0}^{2})}{2}\sum_{n\geq 0}\frac{|(P_{\tau}\varphi'')^{(n)}(0)|}{\sqrt{n!}}|A^{-p}h_1|_0^{n}\\
&\leq&\frac{(|h_1|_{0}^{2}-|A^{-p}h_1|_{0}^{2})}{2}\Big(\sum_{n\geq 0}\tau^n\frac{|(P_{\tau}\varphi'')^{(n)}(0)|^2}{n!}\Big)^{\frac{1}{2}}\Big(\sum_{n\geq 0}\frac{|A^{-p}h_1|_0^{2n}}{\tau^n}\Big)^{\frac{1}{2}}
\end{eqnarray*}
Since $\tau>|A^{-p}h_1|_0^{2}$ the last series is convergent to a constant $C$ depending on $p$ and $h_1$; moreover, from equation (\ref{kuo}) we can write
\begin{eqnarray*}
\sum_{n\geq 0}\tau^n\frac{|(P_{\tau}\varphi'')^{(n)}(0)|^2}{n!}&=&E[(\varphi''(Y))^2].
\end{eqnarray*}
where $Y$ is a Gaussian random variable with mean zero and variance $\tau$. Therefore,
\begin{eqnarray*}
\Vert\varphi(I_1(h_1))-\tilde{\varphi}_p(I_1(h_1))\Vert_{-p}&\leq&\frac{(|h_1|_{0}^{2}-|A^{-p}h_1|_{0}^{2})}{2}\Big(\sum_{n\geq 0}\tau^n\frac{|(P_{\tau}\varphi'')^{(n)}(0)|^2}{n!}\Big)^{\frac{1}{2}}\Big(\sum_{n\geq 0}\frac{|A^{-p}h_1|_0^{2n}}{\tau^n}\Big)^{\frac{1}{2}}\\
&=&C\frac{|h_1|_{0}^{2}-|A^{-p}h_1|_{0}^{2}}{2}E[(\varphi''(Y))^2]^{\frac{1}{2}}\\
&\leq&C\frac{|h_1|_{0}^{2}-|A^{-p}h_1|_{0}^{2}}{2}\sup_{x\in\mathbb{R}}|\varphi''(x)|.
\end{eqnarray*}
The proof is complete.
\end{proof}

\end{document}